\documentclass{ws-ijm}
\usepackage{cite}
\usepackage[all,cmtip]{xy}
\begin{document}

\markboth{Raphael de Omena}{Local invariants on an ICIS}

\catchline{}{}{}{}{}

\title{Bruce-Roberts numbers and indices of vector fields on an ICIS}

\author{Raphael de Omena}

\address{Departamento de Matemática \\ Universidade Federal do Ceará \\ Fortaleza, 60455-760, Ceará, Brazil \\ raphael.marinho@mat.ufc.br}

\maketitle

\begin{abstract}
This work
investigates invariants, including the GSV-index, the local Euler obstruction, and the Brasselet number, within the context of isolated complete intersection singularities (ICIS). The goal is to forge connections among these invariants, facilitating the extraction of both topological and geometric insights through algebraic methods. The Bruce-Roberts number assumes a central role in realizing this objective, given its definition via algebraic tools.
\end{abstract}



\section{Introduction}\label{sec1}

In the field of singularity theory, comprehending and quantifying the properties of singularities are crucial endeavors. One approach involves exploring invariants associated with these singularities. Throughout this work, we elucidate the definitions and properties of local invariants, highlighting their role in characterizing singularities in complex analytic varieties.  

At the core of our inquiry lies the notion of Milnor number. Let \linebreak $(X,0) = (f^{-1}(0),0)$, where $f: (\mathbb{C}^{n+1},0) \rightarrow (\mathbb{C},0)$, be the germ of a hypersurface with an isolated singularity at the origin. Milnor proved in \cite{john1968singular} that the fiber $X_t = f^{-1}(t) \cap \mathbb{B}_{\varepsilon}$, for a sufficiently small ball $\mathbb{B}_{\varepsilon}$ centered at $0$, has the homotopy type of a bouquet of spheres with a real dimension $n = \text{dim}(X)$. The Milnor number is the count of these spheres and is equal to the length of the analytic algebra $\frac{\mathcal{O}_n}{Jf}$, where $\mathcal{O}_n$ denotes the ring of analytic function-germs at the origin and $J(f)$ is the Jacobian ideal of $f$.

In Section \ref{ICIS}, we recall the concept of complete intersection singularities (ICIS). These singularities emerge as a natural extension of hypersurfaces, characterized by being loci of map germs that preserve the expected codimension. For such singularities, the Milnor number is defined for Hamm in \cite{hamm1971lokale}. The Lê-Greul formula \cite{le1974calculation, greuel1975gauss} provides an algebraic expression for computing the Milnor number, concepts that we revisit.

The Milnor number is known to have different interpretations. It can represent the number of Morse points in a Morsification of $f$, as well as the Poincaré-Hopf index of the complex conjugate of the gradient vector field of $f$. This investigation not only provides topological insights into the singularity but also offers algebraic and geometric information.

Various authors have proposed invariants that extend beyond the Milnor number, providing valuable insights into the described phenomena. In Section \ref{Indices}, we revisit the Indices of Vector Fields, specifically focusing on the GSV-index and the Euler obstruction, which encode geometric information. The authors in \cite{de2024milnor} establish connections between some of these proposed invariants, results that we recall for the case of ICIS.

Furthermore, certain generalizations aim to uncover insights when a function is defined over a singular complex space germ $(X,0)$. This is exemplified by the Bruce-Roberts number, introduced in \cite{bruce1988critical} and revisited in Section \ref{BRnumber}. Defined through an algebraic approach, this number allows us to utilize the \texttt{Singular} software \cite{DGPS} for computation. Our primary aim is to establish connections between the discussed invariants, providing a means to extract topological and geometric information by computing them via the \texttt{Singular} program.

\section{The Milnor number of an ICIS} \label{ICIS}

Complete intersections represent a natural extension of hypersurfaces in the following manner: a hypersurface singularity $(X, 0)$ is a germ set in $\mathbb{C}^{n+1}$ defined as the  locus where a single non-constant holomorphic function $f : (\mathbb{C}^{n+1}, 0) \rightarrow (\mathbb{C}, 0)$ vanishes. This condition implies that $(X,0)$ must possess dimension $n$. Now, if we consider $k$ holomorphic functions $f = ( f_1, \ldots, f_k ): (\mathbb{C}^{n+k}, 0) \rightarrow (\mathbb{C}^{k}, 0)$, it is generally not the case that the zero locus $(X,0)$ of $f$ also has dimension $n$. However, when this does occur, we refer to $(X,0)$ as a complete intersection.

\begin{definition}
    A complex space germ $(X, 0) \subset (\mathbb{C}^n,0)$ is termed a \textbf{complete intersection} if the minimal number of generators of its ideal $\mathcal{I}(X,0)$ in $\mathcal{O}_n$ is $n - dim(X,0)$. An \textbf{isolated complete intersection singularity} (ICIS) refers to a complete intersection $(X, 0)$ such that its singular locus, denoted as $Sing(X)$, coincides with the origin.
\end{definition}

A comprehensive understanding of a singularity involves observing its transformation into less degenerate types under deformation. This  requires examining a continuous family of varieties parametrized, where the original singularity emerges at a certain parameter value. Essentially, to comprehend the essence of a singularity, one must explore a spectrum of perspectives, typically represented by a 1-parameter family. A semiuniversal deformation, a special type of deformation, is characterized by its local simplicity. Roughly speaking, it is a deformation that is universal among those exhibiting the same first-order behavior.

To facilitate our discussion, let $(X,0) \subset (\mathbb{C}^n,0)$ denote an ICIS, with its defining ideal $\mathcal{I}(X,0) = {\langle f_1,\ldots, f_{k} \rangle}$, and $\{e_i\}_{i =1}^n$ representing the canonical basis of $\mathbb{C}^n$. The \textit{Tjurina number} corresponds to the dimension of the base of a semiuniversal deformation of $(X,0)$ and on an ICIS can be computed as follows. This equality is established as Theorem 1.6 in \cite{greuel2007introduction}.
\begin{equation*}
\tau(X,0) = \dim_{\mathbb{C}}\dfrac{\mathcal{O}^k_n}{{\langle df(e_i) \rangle}_{i=1}^n + \mathcal{I}(X,0)\mathcal{O}_n^k}.
\end{equation*}

The investigation into whether a singularity $(X, 0)$ is smoothable holds particular interest for several reasons, notably because the smooth nearby fiber serves as a crucial topological entity associated with the singularity. For any given smoothing, a smooth Milnor fiber emerges, a fundamental topological construct that has received significant attention in research. It is worth noting  that since the base space of the semiuniversal deformation of an ICIS is smooth, there exists a unique Milnor fiber up to diffeomorphism \cite{kas1972versal}, a definition we reiterate here.

\begin{definition}
    Let $(X, 0) = (f^{-1}(0), 0)$, $f : (\mathbb{C}^n, 0) \rightarrow (\mathbb{C}^k, 0)$, be an ICIS. For a small enough $\varepsilon$, the fiber
    \begin{equation*}
        {X}_t = f^{-1}(t) \cap \mathbb{B}_{\varepsilon},
    \end{equation*}
    where $0< |t| <\varepsilon$, and $\mathbb{B}_{\varepsilon}$ is the ball centred in $0$ and radius $\varepsilon$, is called the \textbf{Milnor fiber}.
\end{definition}

In this case, Hamm \cite{hamm1971lokale} establishes that the Milnor fiber of the ICIS $(X,0)$ possesses the homotopy type of a bouquet of spheres of same dimension $dim(X,0) = n-k$, extending Milnor's results \cite{john1968singular}, which focused on hypersurfaces with isolated singularity. The middle Betti number of the Milnor fiber, denoted as $b_n(X_t)$ holds particular significance. 

\begin{definition}
The \textbf{Milnor number} of $(X, 0)$, denoted as $\mu(X, 0):= b_n(X_t)$, representing the number of these spheres. 
\end{definition}

An algebraic formula for $\mu(X, 0)$ has been independently provided by Greul \cite{greuel1975gauss} and Lê \cite{le1974calculation}, which is commonly referred to as the Lê-Greul formula. Let \linebreak $f= (f_1, \ldots, f_k): (\mathbb{C}^n,0) \rightarrow (\mathbb{C}^k,0)$ be a germ of holomorphic mapping defining an ICIS $(X,0) = (f^{-1}(0),0)$, and let $g = (f_1, \ldots, f_{k-1})$. Denote $(Y,0) = (g^{-1}(0),0)$ and $J_k(f)$ as the ideal generated by all $k \times k$ minors of the Jacobian matrix of $f$. Then,
\begin{equation*}
\mu(X,0) + \mu(Y,0) = \dim_{\mathbb{C}} \dfrac{\mathcal{O}_n}{\langle f_1,\ldots, f_{k-1}, J_k(f) \rangle}.
\end{equation*}

\section{Indices of vector fields} \label{Indices}

The investigation of vector fields at isolated singularities lies in understanding their most basic invariant: the local Poincaré–Hopf index. This concept has been extensively explored from various angles, resulting in a large body of research. However,  in a singular variety, a natural question arises: What defines a suitable measure of index for vector fields? Answering this question depends on the specific aspects of the local index that we wish to understand. In this section, we outline some of these explorations. For further details, refer to \cite{brasselet2009vector}.

\subsection{The GSV-index}

An essential characteristic of the local Poincaré–Hopf index is its stability under perturbations. In other words, if we approximate the given vector field by another vector field that only has Morse singularities, then the local index of the initial vector field equals the number of singularities of its Morsification counted with signs. Establishing an index for vector fields on singular varieties that possesses this property leads to the GSV-index, introduced by Gómez-Mont, Seade and Verjovsky in \cite{verjovsky1987index}.The GSV-index is defined on an ICIS as follows.

Let $f = (f_1, \ldots, f_k): (\mathbb{C}^n,0) \rightarrow (\mathbb{C}^k,0)$ be a holomorphic mapping germ defining an ICIS $(X,0) = (f^{-1}(0),0)$. As $0$ is an isolated singularity of $X$, the complex conjugate gradient vector fields ${\overline{\nabla} f_1, \ldots, \overline{\nabla}f_k}$ are linearly independent away from the origin. 
\begin{equation*}
\overline{\nabla}f_i(x):=\left(\overline{\frac{\partial f_i}{\partial x_1}}(x),\ldots, \overline{\frac{\partial f_i}{\partial x_n}}(x)\right).
\end{equation*}

Consider the set $\{v, \overline{\nabla} f_1, \ldots, \overline{\nabla} f_k\}$, which forms a $(k + 1)$-frame on $X \setminus \{0\}$, where $v$ is a continuous vector field on $X$, singular only at the origin.

If $K = X \cap \mathbb{S}_{\varepsilon}$ represents the link of $0$ in $X$, then this link constitutes an oriented real manifold of dimension $2n - 1$. The frame described above defines a continuous map $\varphi_v = (v, \overline{\nabla}f_1, \ldots, \overline{\nabla}f_k) : K \rightarrow W^{k+1}(n+k)$, where $W^{k+1}(n+k)$ denotes the Stiefel manifold of complex $(k+1)$-frames in $\mathbb{C}^{n+k}$. Notably, its first non-zero homology group is isomorphic to $\mathbb{Z}$ in dimension $2n - 1$. The map $\varphi_v$ possesses a well-defined degree $deg(\varphi_v) \in \mathbb{Z}$, determined via the induced homomorphism $H_{2n-1}(K) \rightarrow H_{2n-1}(W^{k+1}(n + k))$ in the usual manner.

\begin{definition}
    The GSV-index of $v$ at $0 \in X$, $ind_{GSV}(v,X,0)$ is the degree of the map $\varphi_v$.
\end{definition}

Utilizing the Transversal Isotopy Lemma, there exists an ambient isotopy of the sphere $\mathbb{S}_{\varepsilon}$ that transforms $K$ onto $\partial X_t$, where $\partial X_t$ is the boundary of the Milnor fiber of $(X,0)$. Consequently, $v$ can be interpreted as a nonsingular vector field on $\partial X_t$. In this context, the degree of $\varphi_v$ is connected to the obstruction of extending $v$ into a tangent vector field on $X_t$, which corresponds to the Poincaré-Hopf index of the Milnor fiber.

\begin{proposition}[Theorem 3.2.1, \cite{brasselet2009vector}]
    The GSV-index of $v$ at $0$ equals to the Poincaré-Hopf index of $v$ in the Milnor fiber:
    \begin{equation*}
        ind_{GSV}(v,X,0) = ind_{PH}(v, X).
    \end{equation*}
\end{proposition}

Another noteworthy property of the Poincaré–Hopf index on a germ of a holomorphic vector field $v = \sum_{i=1}^{n} h_i \frac{\partial}{\partial z_i}$ on $\mathbb{C}^{n+1}$ with an isolated singularity at $0$ is that the local index equals the integer:

\begin{equation*}
    \dim_{\mathbb{C}} \dfrac{\mathcal{O}_n}{\left<h_1, \dots, h_n\right>},
\end{equation*}
where $\left<h_1, \dots, h_n\right>$ is the ideal generated by the components of $v$ in the ring $\mathcal{O}_n$. When considering $v$ as the complex conjugate of the gradient vector field of a function $f : (\mathbb{C}^{n+1}, 0) \rightarrow (\mathbb{C}, 0)$, it follows that the Poincaré-Hopf index of $v$ aligns with the Milnor number associated with the hypersurface defined by $f$. Hence, the GSV-index can be interpreted as an extension or generalization of the concept of the Milnor number.

Moreover, this characterization motivates the search for algebraic formulas for the index of vector fields on singular varieties. The homological index, defined by Gómez-Mont in \cite{gomez1998algebraic},  provides a solution to this inquiry. It considers an isolated singularity germ $(X, 0)$ of arbitrary dimension, accompanied by a holomorphic vector field on $X$ that is singular only at $0$. Within this context, we employ Kähler differentials on $X$ and construct a Koszul complex $(\Omega^\bullet_{X, 0}, v)$:
\begin{equation*}
  0 \rightarrow \Omega^n_{X, 0}\rightarrow \Omega^{n-1}_{X, 0} \rightarrow \cdots \rightarrow \mathcal{O}_{X,0} \rightarrow 0,   
\end{equation*}
where the arrows denote contraction of forms by the vector field $v$. The homological index of $v$ is then defined as the Euler characteristic of this complex.

\begin{definition}
    The \textbf{homological index}, denoted as $ind_{hom}(v, X, 0)$, of the holomorphic vector field $v$ on $(X, 0)$ is defined as:

\begin{equation*}
    ind_{hom}(v, X, 0) = \sum_{i=0}^{n} (-1)^i h_i(\Omega^\bullet_{X, 0}, v),
\end{equation*}
where $h_i(\Omega^\bullet_{X, 0}, v)$ represents the dimension of the $i$-th homology group of the complex as a vector space over $\mathbb{C}$.

\end{definition}

The authors in \cite{bothmer2008algebraic} have demonstrated that the homological index corresponds to the GSV-index for an ICIS. With this understanding, we refer to the GSV-index in our discussion. It is worth noting that while the homological index is defined for vector fields on arbitrary isolated normal singularity germs, the GSV-index is specifically defined on complete intersection germs.

\subsection{The local Euler obstruction}

The notion of the local Euler obstruction was initially pioneered by MacPherson \cite{macpherson1974chern}, serving as a fundamental component in the development of characteristic classes for singular complex algebraic varieties. An alternative definition was proposed by Brasselet and Schwartz \cite{brasselet1981classes}, employing vector fields, which places the local Euler obstruction within the context of indices of vector fields on singular varieties. This approach aligns well with our present investigation.

Let $(X,0)\subset(\mathbb{C}^n,0)$ denote an equidimensional reduced complex analytic germ of dimension $d$ in an open set $U\subset\mathbb{C}^n$. Here, $G(d,n)$ represents the Grassmannian manifold and we consider the Gauss map $\phi: X_{reg}\rightarrow U\times G(d,n)$, defined as \linebreak $x\mapsto(x,T_x(X_{reg})).$

\begin{definition}
The closure of the image of the Gauss map $\phi$ in $U\times G(d,n)$, denoted by $\tilde{X}$, is termed the \textbf{Nash modification} of $X$. This constitutes a complex analytic space equipped with a holomorphic projection map $\nu:\tilde{X}\rightarrow X$.
\end{definition}

We proceed to consider the extension of the tautological bundle $\mathcal{T}$ over $U\times G(d,n)$. Since $\tilde{X}\subset U\times G(d,n)$, we denote by $\widetilde {\mathcal{T}}$ the restriction of $\mathcal{T}$ to $\tilde{X}$, referred to as the \textit{Nash bundle}. Brasselet and Schwartz \cite{brasselet1981classes} demonstrated that each stratified vector field $v$, which is non-null on a subset $A \subset X$, possesses a canonical lifting to a non-null section $\tilde{v}$ of the Nash bundle $\widetilde {\mathcal{T}}$ over $\nu^{-1}(A) \subset \widetilde{X}$.

Now, let us consider a stratified radial vector field $v(x)$ in a neighborhood \linebreak 
 $\mathbb{S}_{\varepsilon} := \partial{\mathbb{B}_{\varepsilon}}$ of $\{0\}$ in $X$ and let $\tilde{v}$ be its lifting on $\nu^{-1}(X \cap \mathbb{S}_{\varepsilon})$ to a section of the Nash bundle. Denote $\mathcal{O}{(\tilde{v})} \in H^{2d}(\nu^{-1}(X \cap
{\mathbb{B}_{\varepsilon}}), \nu^{-1}(X \cap {\mathbb{S}_{\varepsilon}}), \mathbb{Z})$ as the
obstruction cocycle for extending $\tilde{v}$ as a nowhere-zero section of
$\widetilde {\mathcal{T}}$ inside $\nu^{-1}(X\cap {\mathbb{B}_{\varepsilon}})$.

\begin{definition}
    The \textbf{local Euler obstruction} ${\rm Eu}_{X}(0)$ is defined as the evaluation of the cocycle $\mathcal{O}(\tilde{v})$ on the fundamental class of the pair $[\nu^{-1}(X \cap {\mathbb{B}_{\varepsilon}}), \nu^{-1}(X \cap {\mathbb{S}_{\varepsilon}})]$, therefore it is an integer.
\end{definition}

The determination of the local Euler obstruction often relies on topological methods, including the application of the Euler characteristic. Brasselet, Lê, and Seade have outlined a formula for this purpose in their work \cite{brasselet2000euler}.

 We turn our attention to a vector field defined by a holomorphic function \linebreak $f:X\rightarrow\mathbb{C}$ with an isolated singularity at the origin, obtained as the restriction of a holomorphic function $F:U\rightarrow\mathbb{C}$. Denote by $\overline{\nabla}F(x)$ the conjugate of the gradient vector field of $F$ at $x\in U$.

Since $f$ has an isolated singularity at the origin, for all $x\in X\setminus \{0\}$, the projection ${\zeta}_i(x)$ of $\overline{\nabla}F(x)$ onto $T_x(V_i(x))$ is nonzero, where $V_i(x)$ is a stratum containing $x$. Utilizing this projection, Brasselet, Massey, Parameswaran, and Seade constructed, in \cite{BMPS}, a stratified vector field on $X$, denoted by $\overline{\nabla}_Xf$. Let $\tilde{\overline{\nabla}}_Xf$ be the lifting of $\overline{\nabla}_Xf$ as a section of the Nash bundle $\widetilde {\mathcal{T}}$ over $\tilde{X}$, without singularity on $\nu^{-1}(X\cap \mathbb{S}_{\varepsilon})$.

\begin{definition}
The \textbf{local Euler obstruction of the function} $f$, denoted as ${\rm Eu}_{f,X}(0)$, is the evaluation of $\mathcal{O}(\tilde{\overline{\nabla}}_Xf)$ on the fundamental class \linebreak $[\nu^{-1}(X\cap \mathbb{B}_{\varepsilon}),\nu^{-1}(X\cap \mathbb{S}_{\varepsilon})]$.
\end{definition}

The local Euler obstruction of a function serves as a generalization of the Milnor number in certain aspects. Seade, Tib{\u{a}}r, and Verjovsky demonstrated in \cite{seade2005milnor} that ${\rm Eu}_{f,X}(0)$, up to sign, corresponds to the number of Morse critical points in a Morsification of $f$. 

Brasselet, Massey, Parameswaran, and Seade \cite{BMPS} also introduce the concept of the \textit{defect} when $f$ may have a non-isolated singularity at the origin, which coincides with the Euler obstruction of $f$ in the case of an isolated singularity. The defect of $f$ is closely related to another invariant defined by Dutertre and Grulha \cite{dutertre2014le}.

The \textit{Brasselet number} for a holomorphic function $f:(X,0)\rightarrow(\mathbb{C},0)$ is defined as follows:
\begin{equation*}
{\rm B}_{f,X}(0)=\sum_{i=1}^{q}\chi(V_i\cap f^{-1}(\delta)\cap \mathbb{B}_{\varepsilon}){\rm Eu}_X(V_i),
\end{equation*}
where $\mathcal{V} = {V_{i}}$ represents a suitable stratification. In particular, if $f$ has an isolated singularity, the Brasselet number simplifies to ${\rm B}_{f,X}(0)= {\rm Eu}_X(0) - {\rm Eu}_{f,X}(0)$.

A natural extension of the local Euler obstruction, applied to a stratified vector field $v$ at an isolated singularity in $X$, can be obtained by considering a stratified $k$-field $v^{(k)} = (v_1, \dots, v_k)$. For further details, refer to \cite{brasselet2009vector}. The class of the obstruction cocycle, intended to extend the lifting $\tilde{v}^{(k)}$ to a $k$-field without singularity over \linebreak $\nu^{-1}(\sigma \cap X)$, is denoted by $\mathcal{O}(\tilde{v}^{(k)}) \in H^{2(d-k+1)}(\nu^{-1}(\sigma \cap X), \nu^{-1}(\partial \sigma \cap X))$.The \textit{local Euler obstruction} ${\rm Eu}_X(v^{(k)},0)$ of the stratified $k$-field $v^{(k)}$ at an isolated singularity at the origin is defined as the integer obtained by evaluating $\mathcal{O}(\tilde{v}^{(k)})$ on the fundamental class $[(\nu^{-1}(\sigma \cap X), \nu^{-1}(\partial \sigma \cap X))]$.

Let $f: X \rightarrow \mathbb{C}^k$ be a holomorphic map with singular set $Sing(f)$. Grulha \cite{junior2007obstruccao} constructs a stratified $k$-field $\overline{\nabla}_X^{(k)}f$ over $\mathbb{S}_{\varepsilon}(0) \cap X \setminus Sing(f)$, without singularity. If $\sigma$ is a cell with barycenter $0$ such that $\sigma \cap Sing(f) = \emptyset$, then the $k$-field admits a lifting $\tilde{\overline{\nabla}}_X^{(k)}f$ as a section of $\widetilde {\mathcal{T}}$ over $\nu^{-1}(\partial \sigma \cap X)$. 

\begin{definition}
Let $\sigma$ be a generic cell, and let $\overline{\nabla}_X^{(k)}f$ be as above. The \textbf{Euler obstruction of the map} $f$ at the origin is ${\rm Eu}_{f,X}(0) = {\rm Eu}_X(\overline{\nabla}_X^{(k)}f, 0)$.
\end{definition} 

Brasselet, Grulha, and Ruas \cite{brasselet2010euler}  demonstrated that the Euler obstruction of a map remains independent of a generic choice of $\sigma$. Grulha, Santana, and Ruiz \cite{grulha2022geometrical} express the Euler obstruction of the holomorphic map-germ $f:(X,0) \rightarrow (\mathbb{C}^2,0)$ in terms of the number of critical points of a suitable Morsification. Building upon these findings, the authors in \cite{de2024milnor} established connections between the discussed invariants in the context of an isolated determinantal singularity, which represents a generalization of an ICIS. To elucidate, we restate this as follows.

Consider $f = (f_1,f_2): (X,0) \rightarrow (\mathbb{C}^2,0)$ a germ of holomorphic mapping, which is the restriction of $F = (F_1,F_2) :(U,0) \rightarrow (\mathbb{C}^2,0)$, where $U \subset \mathbb{C}^{n+1}$ is an open set. We denote by $H = f^{-1}_ 2(0)$ and by $H_{\delta} = f^{-1}_ 2(\delta)$.

\begin{theorem}[\cite{de2024milnor}] \label{invariantsofmap}
Let $(X,0) \subset (\mathbb{C}^{n+1},0)$ be an ICIS of dimension $d+1$, with $d> 1$, and let $f: (X,0) \rightarrow (\mathbb{C}^2,0)$ be as above, such that ${F_2}_{\lvert{X}} = f_2$ is linear, and ${F_1}_{\lvert{Y}} = f_1$ has an isolated singularity at the origin, where $Y = X \cap H$. Then, the Euler obstruction ${\rm Eu}_{f, X\cap H_{\delta}}(x)$ of the map $f$ at a point $x \in X \cap H_{\delta}$ is equal to the following equivalent expressions:
\begin{itemize}
\item[(a)] $ind_{GSV}(\nabla f_1,Y,0)$;
\item[(b)] $\mu(Y,0) + \mu(Y \cap f^{-1}_1(0),0)$;
\item[(c)] $\mu(Y \cap f^{-1}_1(0),0) + (-1)^d({\rm Eu}_X(0) + 1);$
\item[(d)] $\mu(Y \cap f^{-1}_1(0),0) + m_d(Y) + (-1)^d({\rm Eu}_Y(0) -1)$;
\item[(e)] $\mu(Y,0) + \mu(Y \cap p^{-1}(0),0) + (-1)^d {\rm Eu}_{f_1,Y}(0)$, where $p: \mathbb{C}^{n} \rightarrow \mathbb{C}$ is a generic linear function with respect to $Y$;
\item[(f)] $m_d(Y) + (-1)^d {\rm Eu}_{f_1,Y}(0)$;
\item[(g)] $\mu(Y,0) + (-1)^{d-1}({\rm B}_{f_1,Y}(0) -1)$.
\end{itemize}
\end{theorem}

In the above theorem, $m_d(X)$ represents the \textit{$d$-th polar multiplicity}, as defined by Gaffney \cite{gaffney1993polar}. If $X$ possesses a unique smoothing, then the polar multiplicity depends exclusively on $X$,  making it an invariant of the analytic variety $X$.

\begin{remark}
    As previously mentioned, Theorem \ref{invariantsofmap} was originally formulated for isolated determinantal singularities. In this case, the so-called Poincaré-Hopf-Nash index is well-defined, which, in the case of ICIS, coincides with the GSV-index.
\end{remark}

\section{The Bruce-Roberts number} \label{BRnumber}

The Milnor number is one of the key invariants in singularity Theory due to its good properties, as discussed before. Besides, it is known that $\mu(f)$ is finite if, and only if, $f$ has an isolated singularity, which is equivalent to $f$ being $\mathcal{R}$-finitely determined, where $f: (\mathbb{C}^n,0) \rightarrow (\mathbb{C},0)$ is a germ of holomorphic function. Here, $\mu(f)$ denotes the Milnor number of $(V(f),0) = (f^{-1}(0),0)$.

In this sense, there exists an invariant that maintains the aforementioned property of the Milnor number on germs of functions defined on singular complex space germs $(X,0) \subset (\mathbb{C}^n,0)$, namely, its finiteness is equivalent to $f$ having an isolated singularity on the analytic space $(X,0)$, which in turn is equivalent to $f$ being $\mathcal{R}_X$-finitely determined, where $\mathcal{R}_X$ is the subgroup of $\mathcal{R}$ consisting of coordinate changes preserving $X$. This is the case of the number introduced by Bruce and Roberts in \cite{bruce1988critical}.

Let $(X,0) \subset (\mathbb{C}^n,0)$ be a germ of reduced analytic variety, and let $\Theta_{n}$ be the \linebreak $\mathcal{O}_n$-module of germs of vector fields on $(\mathbb{C}^n,0)$. The set of germs of vector fields in $\mathbb{C}^n$ that are tangent to $(X,0)$ is defined by
\begin{equation*}
\Theta_{X} := \{\xi \in \Theta_{n}; dh(\xi) \in \mathcal{I}(X,0), \forall h \in \mathcal{I}(X,0) \}.
\end{equation*}

If we consider $X = \mathbb{C}^n$, then $\Theta_{X} = \Theta_{n}$. In this case, the image of $\Theta_{X}$ under the differential $df$ of a germ of holomorphic function $f: (\mathbb{C}^n,0) \rightarrow (\mathbb{C},0)$ is given by $df(\Theta_{X}) = Jf$. In this way, the subsequent definition extends the concept of the Milnor number.

\begin{definition}
Let $(X,0) \subset (\mathbb{C}^n,0)$ be a germ of reduced analytic variety and $f \in \mathcal{O}_n$. The \textbf{Bruce-Roberts number} of $f$ with respect to the variety $(X,0)$ is defined as
\begin{equation*}
\mu_{BR}(f,X):= dim_{\mathbb{C}}\dfrac{\mathcal{O}_n}{df(\Theta_{X})}.
\end{equation*}

The \textbf{relative Bruce-Roberts number} of $f$ with respect to $(X,0)$ is defined as 
\begin{equation*}
    \overline{\mu_{BR}}(f,X):= \dim_{\mathbb{C}}\dfrac{\mathcal{O}_n}{df(\Theta_{X}) + \mathcal{I}(X,0)}.
\end{equation*}
\end{definition}

Even though there is no explicit form for the generators of $\Theta_X$ when $(X, 0)$ is an analytic variety, we can utilize software tools like \texttt{Singular} to compute these generators. Below is the step-by-step process, as formulated in \cite{lima2022numero}.

Let us assume that the ideal $\mathcal{I}(X,0)$ is generated by $f_1, \ldots, f_k$ germs of functions in $(\mathbb{C}^n, 0)$. Our goal is to find vector fields $\xi \in \Theta_{n}$ such that:

\begin{equation*}
\sum_{i=1}^{n} \xi_i \frac{\partial f_r}{\partial x_i} - \lambda_{1r}f_1 - \cdots - \lambda_{kr}f_k = 0.
\end{equation*}
for every $r = 1, \ldots, k$.

For each $r = 1, \ldots, k$, denote
\begin{equation*}
T_r = \text{syz}\left( \left< \frac{\partial f_r}{\partial x_1}, \ldots, \frac{\partial f_r}{\partial x_n}, f_1, \ldots, f_k \right>\right),
\end{equation*}
where $\text{syz}$ denotes the syzygy module, which can be computed using \texttt{Singular}.

Note that $T_r$ is a submodule of $\mathcal{O}^{n+k}_n$. By projecting its first $n$ coordinates, we obtain a submodule of $\Theta_{n}$ denoted by $T_{rn}$. To determine the generators of $\Theta_X$, it suffices to take the intersection $\bigcap_{i=1}^k T_{in}$.

In \cite{lima2023bruce}, the authors investigate the Bruce-Roberts number and the relative Bruce-Roberts number of a germ of function $f: (\mathbb{C}^n,0) \rightarrow (\mathbb{C},0)$ over an ICIS \linebreak $(X,0) \subset (\mathbb{C}^n,0)$, relating them to the Milnor number and the Tjurina number. For convenience, we state them here.

\begin{theorem}[\cite{lima2023bruce}] \label{BruceRoberts} Let $(X,0)$ be an ICIS and $f \in \mathcal{O}_n$ be $\mathcal{R}_X$-finitely determined. Then,
\begin{equation*}
\mu_{BR}(f,X) = \mu(f) + \mu(X \cap f^{-1}(0),0) + \mu(X,0) - \tau(X,0) - \dim_{\mathbb{C}}\dfrac{\mathcal{O}_n}{Jf + \mathcal{I}(X,0)} + \dim_{\mathbb{C}}\dfrac{\mathcal{I}(X,0) \cap Jf}{\mathcal{I}(X,0)Jf}.
\end{equation*}
\end{theorem}

Considering the relative Bruce-Roberts number, the above expression is simplified to:

\begin{theorem}[\cite{lima2023bruce}] \label{BRICIS}Let $(X,0)$ be an ICIS and $f \in \mathcal{O}_n$ such that $\overline{\mu_{BR}}(f,X)< \infty$. Then,
\begin{equation*}
\overline{\mu_{BR}}(f,X) = \mu(X \cap f^{-1}(0),0) + \mu(X,0) - \tau(X,0).
\end{equation*}
\end{theorem}

Based on these results, we can establish connections with other discussed invariants. In the literature, connections between the Bruce-Roberts number and the local Euler obstruction are documented, as in \cite{grulha2009euler}, and with the Chern Obstruction as demonstrated in \cite{lima2023relative}. The following proposition enables us to compute these invariants using the algebraic formulas provided by the Bruce-Roberts number and the relative Bruce-Roberts number.

\begin{proposition}
Let $f: (\mathbb{C}^n,0) \rightarrow (\mathbb{C},0)$ be a germ of holomorphic function with isolated singularity over the ICIS $(X,0) \subset (\mathbb{C}^n,0)$ of dimension $d > 2$. Then, the relative Bruce-Roberts number $\overline{\mu_{BR}}(f,X)$ is equal to the following equivalent expressions:
\begin{itemize}
\item[(a)] $ind_{GSV}(\nabla f,X,0) - \tau(X,0)$;
\item[(b)] $\mu(X \cap f^{-1}(0),0) + m_d(X) + (-1)^d({\rm Eu}_X(0) -1) - \tau(X,0)$;
\item[(c)] $\mu(X,0) + \mu(X \cap p^{-1}(0),0) + (-1)^d {\rm Eu}_{f,X}(0) - \tau(X,0)$, where $p: \mathbb{C}^n \rightarrow \mathbb{C}$ is a linear generic function with respect to $X$;
\item[(d)] $m_d(X) + (-1)^d {\rm Eu}_{f,X}(0) - \tau(X,0)$;
\item[(e)] $\mu(X,0) + (-1)^{d-1}({\rm B}_{f,X}(0) -1) - \tau(X,0)$.
\end{itemize}
\end{proposition}
\begin{proof}
   Consider a $1$-parameter deformation of $(X,0)$ given by a family $(\mathfrak{X},0)$ of smoothings, with the projection $\pi: (\mathfrak{X},0) \rightarrow (\mathbb{C},0)$.

\[
\xymatrix{
(X,0) \ar[d] \ar@{^{(}->}[r] & (\mathfrak{X},0) \ar[d]^{\pi} \ar@{^{(}->}[r] & (\mathbb{C}^n \times \mathbb{C},0) \ar[dl] \\
\{0\} \ar@{^{(}->}[r] & (\mathbb{C},0) &
}
\]

Let $\tilde{f}: (\mathfrak{X},0) \rightarrow (\mathbb{C},0)$ be a germ of holomorphic function such that $\tilde{f}_{\lvert{X}} = {f}_{\lvert{X}}$. Now, consider the map $g = (\tilde{f}, \pi): (\mathfrak{X},0) \rightarrow (\mathbb{C}^2,0)$. The Theorem \ref{invariantsofmap} guarantees that the Euler obstruction ${\rm Eu}_{g, \mathfrak{X} \cap \pi^{-1}(\delta)}(x)$ of $g$ at a point $x \in \mathfrak{X} \cap \pi^{-1}(\delta)$ is given by
\begin{equation*}
{\rm Eu}_{g, \mathfrak{X} \cap \pi^{-1}(\delta)}(x) = \mu(X,0) + \mu(X \cap f^{-1}(0),0).
\end{equation*}

Therefore, Theorem \ref{BRICIS} ensures that
\begin{equation*}
\overline{\mu_{BR}}(f,X) = {\rm Eu}_{g, \mathfrak{X} \cap \pi^{-1}(\delta)}(x) - \tau(X,0).
\end{equation*}

The desired expressions follow directly from the relations obtained in Theorem \ref{invariantsofmap}. 
\end{proof}

Using Theorem \ref{BruceRoberts}, we obtain analogous expressions for the Bruce-Roberts number.

\begin{proposition}
Let $f: (\mathbb{C}^n,0) \rightarrow (\mathbb{C},0)$ be a germ of holomorphic function with isolated singularity over the ICIS $(X,0) \subset (\mathbb{C}^n,0)$ of dimension $d > 2$. Then, the Bruce-Roberts number $\mu_{BR}(f,X)$ is equal to the following equivalent expressions:
\begin{itemize}
\item[(a)] $\mu(f) + ind_{GSV}(\nabla f,X,0) - \tau(X,0) - dim_{\mathbb{C}}\dfrac{\mathcal{O}_n}{Jf + \mathcal{I}(X,0)} + dim_{\mathbb{C}}\dfrac{\mathcal{I}(X,0) \cap Jf}{\mathcal{I}(X,0)Jf}$;
 
\item[(b)] $\mu(f) + \mu(X \cap f^{-1}(0),0) + m_d(X) + (-1)^d({\rm Eu}_X(0) -1) - \tau(X,0) - dim_{\mathbb{C}}\dfrac{\mathcal{O}_n}{Jf + \mathcal{I}(X,0)} + dim_{\mathbb{C}}\dfrac{\mathcal{I}(X,0) \cap Jf}{\mathcal{I}(X,0)Jf}$;
 
\item[(c)] $\mu(f) + \mu(X,0) + \mu(X \cap p^{-1}(0),0) + (-1)^d {\rm Eu}_{f,X}(0) - \tau(X,0) - dim_{\mathbb{C}}\dfrac{\mathcal{O}_n}{Jf + \mathcal{I}(X,0)} + dim_{\mathbb{C}}\dfrac{\mathcal{I}(X,0) \cap Jf}{\mathcal{I}(X,0)Jf}$, where $p: \mathbb{C}^n \rightarrow \mathbb{C}$ is a linear generic function with respect to $X$;
 
	\item[(d)] $\mu(f) + m_d(X) + (-1)^d{\rm Eu}_{f,X}(0) - \tau(X,0) - dim_{\mathbb{C}}\dfrac{\mathcal{O}_n}{Jf + \mathcal{I}(X,0)} + dim_{\mathbb{C}}\dfrac{\mathcal{I}(X,0) \cap Jf}{\mathcal{I}(X,0)Jf}$;
	\item[(e)] $\mu(f) + \mu(X,0) + (-1)^{d-1}({\rm B}_{f,X}(0) -1) - \tau(X,0) - dim_{\mathbb{C}}\dfrac{\mathcal{O}_n}{Jf + \mathcal{I}(X,0)} + dim_{\mathbb{C}}\dfrac{\mathcal{I}(X,0) \cap Jf}{\mathcal{I}(X,0)Jf}$.
\end{itemize}
\end{proposition}

In the context of hypersurfaces, the authors of \cite{nuno2020bruce} and \cite{kourliouros2021milnor} independently presented relationships among these Milnor numbers. Combining the findings, we additionally observe:

\begin{corollary}
    If $(X,0) \subset (\mathbb{C}^n,0)$ is a hypersurface singularity of dimension $d>2$, then the Bruce-Roberts number $\mu_{BR}(f,X)$ is equal to the following equivalent expressions:
\begin{itemize}
\item[(a)] $\mu(f) + ind_{GSV}(\nabla f,X,0) - \tau(X,0)$;
\item[(b)] $\mu(f) + \mu(X \cap f^{-1}(0),0) + m_d(X) + (-1)^d({\rm Eu}_X(0) -1) - \tau(X,0)$;
\item[(c)] $\mu(f) + \mu(X,0) + \mu(X \cap p^{-1}(0),0) + (-1)^d {\rm Eu}_{f,X}(0) - \tau(X,0)$, where $p: \mathbb{C}^n \rightarrow \mathbb{C}$ is a linear generic function with respect to $X$;
\item[(d)] $\mu(f) + m_d(X) + (-1)^d {\rm Eu}_{f,X}(0) - \tau(X,0)$;
\item[(e)] $\mu(f) + \mu(X,0) + (-1)^{d-1}({\rm B}_{f,X}(0) -1) - \tau(X,0)$.
\end{itemize}
\end{corollary}

In their work \cite{bivia2022bruce}, the authors delve into the investigation of the so-called \textit{Bruce-Roberts Tjurina number} of the function $f$ concerning the variety $(X,0)$, particularly when there exists a coordinate system such that both $f$ and $X$ are quasihomogeneous with respect to the same positive rational weights. The Bruce-Roberts Tjurina number is formally defined as follows:
\begin{equation*}
\tau_{BR}(f,X):= \dim_{\mathbb{C}}\dfrac{\mathcal{O}n}{df(\Theta_{X}) + \left< f \right>}.
\end{equation*}

This quantity characterizes the infinitesimal deformations of $f$ under \linebreak $\mathcal{K}_X$-equivalence, representing diffeomorphisms that preserve $X$ and allow for the multiplication of $f$ by units in $\mathcal{O}_n$. The authors have shown that the equality \linebreak $\mu_{BR}(f,X) = \tau_{BR}(f,X)$ is indicative of the relative quasihomogeneity of the pair $(f,X)$. In such cases, the Bruce-Roberts Tjurina number can also be computed using the expressions provided above.

\section*{Acknowledgments}

\end{document}